\documentclass[a4paper,12pt]{amsart}

\usepackage[utf8]{inputenc}
\usepackage[english]{babel}
\usepackage{amsfonts}
\usepackage{amsmath}
\usepackage{mathrsfs}
\usepackage{amssymb}
\usepackage{amsthm}
\usepackage{amscd}
\usepackage{latexsym}
\usepackage{amstext}
\usepackage{amsxtra}
\usepackage[alphabetic,lite]{amsrefs} % for bibliography\usepackage{color}
\usepackage{paralist}
\usepackage[colorinlistoftodos]{todonotes}
\usepackage[all]{xy}
\usepackage{enumerate}
\usepackage{multirow}
\usepackage{wasysym}
\usepackage{latexsym} % for \Box
\usepackage{comment}
\usepackage{colonequals}

\usepackage{geometry}
\usepackage[
        draft=false,
        colorlinks, citecolor=darkgreen,
        pdfauthor={F. F. Favale, J.C. Naranjo, G. P. Pirola, S. Torelli},
        pdftitle={Coverings},
        linktocpage        
]{hyperref}
\hypersetup{citecolor=blue,linktocpage}

\newtheorem{theorem}{Theorem}[section]
\newtheorem*{theorem*}{Theorem}
\newtheorem{lemma}[theorem]{Lemma}

\newtheorem{proposition}[theorem]{Proposition}
\newtheorem{remark}[theorem]{Remark}

\newtheorem{example}[theorem]{Example}

 \newtheorem{question}[theorem]{Question}

\newcommand{\nc}{\newcommand}

\nc{\cH}{{\mathcal H}}
\nc{\cA}{{\mathcal A}}
\nc{\cG}{{\mathcal G}}
\nc{\cC}{{\mathcal C}}
\nc{\cD}{{\mathcal D}}
\nc{\cO}{{\mathcal O}}
\nc{\cI}{{\mathcal I}}
\nc{\cB}{{\mathcal B}}
\nc{\cY}{{\mathcal Y}}
\nc{\cK}{{\mathcal K}} 
\nc{\cX}{{\mathcal X}}
\nc{\cS}{{\mathcal S}}
\nc{\cE}{{\mathcal E}}
\nc{\cF}{{\mathcal F}}
\nc{\cZ}{{\mathcal Z}}
\nc{\cQ}{{\mathcal Q}}
\nc{\cN}{{\mathcal N}}
\nc{\cP}{{\mathcal P}}
\nc{\cL}{{\mathcal L}}
\nc{\cM}{{\mathcal M}}
\nc{\cT}{{\mathcal T}}
\nc{\cW}{{\mathcal W}}
\nc{\cU}{{\mathcal U}}
\nc{\cJ}{{\mathcal J}}
\nc{\cV}{{\mathcal V}}
\nc{\cR}{{\mathcal R}}
\nc{\bH}{{\mathbb H}}
\nc{\bA}{{\mathbb A}}
\nc{\bG}{{\mathbb G}}
\nc{\bC}{{\mathbb C}}
\nc{\bO}{{\mathbb O}}
\nc{\bI}{{\mathbb I}}
\nc{\bB}{{\mathbb B}}
\nc{\bY}{{\mathbb Y}}
\nc{\bK}{{\mathbb K}} 
\nc{\bX}{{\mathbb X}}
\nc{\bS}{{\mathbb S}}
\nc{\bE}{{\mathbb E}}
\nc{\bF}{{\mathbb F}}
\nc{\bZ}{{\mathbb Z}}
\nc{\bQ}{{\mathbb Q}}
\nc{\bN}{{\mathbb N}}
\nc{\bP}{{\mathbb P}}
\nc{\bL}{{\mathbb L}}
\nc{\bM}{{\mathbb M}}
\nc{\bT}{{\mathbb T}}
\nc{\bW}{{\mathbb W}}
\nc{\bU}{{\mathbb U}}
\nc{\bD}{{\mathbb D}}
\nc{\bJ}{{\mathbb J}}
\nc{\bV}{{\mathbb V}}
\nc{\bR}{{\mathbb R}}

\nc{\OO}{\mathcal{O}}
\nc{\PP}{\mathbb{P}}

\DeclareMathOperator{\id}{id}
\DeclareMathOperator{\Tr}{Tr}
\DeclareMathOperator{\Nm}{Nm}

\DeclareMathOperator{\Hom}{Hom}

\DeclareMathOperator{\Pic}{Pic}

\DeclareMathOperator{\Sim}{Sym}

\DeclareMathOperator{\Sym}{Sym}

\DeclareMathOperator{\TorM}{\tau}
\DeclareMathOperator{\PrymM}{Pr}
\DeclareMathOperator{\SpinM}{s\tau}

\nc{\fA}{{\mathfrak{A}}}
\nc{\fB}{{\mathfrak{B}}}
\nc{\fC}{{\mathfrak{C}}}
\nc{\fD}{{\mathfrak{D}}}
\nc{\fE}{{\mathfrak{E}}}
\nc{\fF}{{\mathfrak{F}}}

\nc{\p}{\partial}
\nc{\ph}{\hat{\partial}}

\nc{\war}{{\color{red} CHECK}}
\nc{\rtext}[1]{{\color{red}{#1}}}

\author[G. Codogni]{Giulio Codogni}
\address[G. Codogni]{Dipartimento di Matematica, Università di Roma Tor Vergata, Via della Ricerca Scientifica 1, 00133 Roma, Italy}
\email{codogni@mat.uniroma2.it}

\author[V. Gonz\'alez-Alonso]{V\'ictor Gonz\'alez-Alonso}
\address[V. Gonz\'alez-Alonso]{Institut f\"ur Algebraische Geometrie, 
    Leibniz Universit\"at Hannover, 
    Welfengarten 1, 30167 Hannover, Germany}
\email{gonzalez@math.uni-hannover.de}

\author[S. Torelli]{Sara Torelli}
\address[S. Torelli]{Institut f\"ur Algebraische Geometrie, 
    Leibniz Universit\"at Hannover, 
    Welfengarten 1, 30167 Hannover, Germany}
\email{torelli@math.uni-hannover.de}



\title{Rigidity of modular morphisms via Fujita decomposition}

 \thanks{GC would like to thank Julius Ross, who nicely raised the question about the rigidity of the Torelli map about ten years ago. We hope he will enjoy this late answer. We thank Benson Farb for private communications. We also thank the
referee and the editors for their very valuable comments. GC is partially supported by  the MIUR Excellence Department Project MatMod@TOV, CUP E83C23000330006, awarded to the Department of Mathematics, University of Rome Tor Vergata, and he is a member of the GNSAGA group of INdAM. ST is supported by the Alexander von Humboldt Foundation.}
\begin{document}

\begin{abstract}
    We prove that the Torelli, Prym  and Spin-Torelli morphisms, as well as covering maps between moduli stacks of smooth projective curves cannot be deformed. The proofs use properties of the Fujita decomposition of the Hodge bundle of families of curves.
\end{abstract}

\maketitle

\section{introduction}\label{sec:intro}

%In this paper, we establish infinitesimal rigidity for some modular morphisms from a moduli stack of curves with possible additional structure, and target either a moduli stack of curves or a moduli stack of principally polarised abelian varieties. 
%Our results provide partial answers to the general expectation that moduli stacks of curves with possible additional structure and moduli stack of morphisms from the moduli stack of curves are rigid. This expectation has been confirmed in several cases in recent years but is still an open problem in full generality. Later in the introduction, we will give more details on results and open problems, relating them to our contribution.

Let $\mathcal{M}$ and $\mathcal{A}$ be Deligne-Mumford stacks. A non-constant morphism $\mathcal{M}\to \mathcal{A}$ is {\em globally rigid} if it is the unique non-constant morphism between $\mathcal{M}$ and $\mathcal{A}$, {\em locally rigid} if it does not admit non-trivial local deformations with fixed target and domain, and {\em infinitesimally rigid} if it does not admit non-trivial first-order deformations with fixed target and domain. In particular, global and infinitesimal rigidity both imply local rigidity. However, they do not imply each other, as certain first-order deformations may not extend to local deformations, and a discrete set of morphisms may all have no first-order deformations. From the point of view of the corresponding moduli stack of non-constant morphisms with fixed target and domain, global rigidity forces the moduli to be just one point, while infinitesimal rigidity determines whether the point is reduced.

\medskip

%Our first result concerns infinitesimal rigidity of the Torelli morphism $\TorM:\mathcal{M}_g\to \mathcal{A}_g$, from the moduli stack of smooth projective curves  $\mathcal{M}_g$ to the moduli stack of principally polarised abelian varieties $\mathcal{A}_g$. %This morphism realizes NON e' UNA IMMERSIONE an immersion of $\mathcal{M}_g$ in $\mathcal{A}_g$, defined at the level of points by mapping a smooth projective curve $C$ to its Jacobian variety $J(C)$ with its natural principal polarization $\Theta_C$. 

In Section \ref{sec:proof1}, we prove the following.

\begin{theorem}\label{thm:main1} For any $g\geq 3$ the Torelli morphism $\TorM:\cM_g\to \cA_g$, from the moduli stack of genus $g$ smooth projective curves  $\mathcal{M}_g$ to the moduli stack of principally polarised abelian $g$-folds, is infinitesimally rigid. \end{theorem}

In \cite{Far21}, Farb can prove the global rigidity, i.e. uniqueness, for the Torelli morphism for $g\geq 3$ in the category of complex orbifolds (as in \cite{Far21}*{Remark 2.1}). Following his proof one obtains also uniqueness in the category of stacks. Uniqueness as a map between coarse moduli spaces is  still open.
%\textcolor{red}{lo statement dell'articolo di Farb é per orbifold globali, cioè mappe dallo spazio di Teichmüller al Siegel upper half space \textbf{compatibile con un omomorfismo di gruppi} dallo mapping class group allo Siegel group. Questo sarebbe più o meno lo stesso che provare l'unicità fra come mappa fra stacks. Fra spazi coarse ci potrebbero essere altre mappe analitiche che non vengono da una mappa olomorfa \textbf{globale} da Teichmüller a Siegel compatibile coi gruppi \textbf{tramite un omomorfismo}.}

 Let us recall that the infinitesimal rigidity of $\mathcal{M}_g$ is still unknown. To the best of our knowledge, the most recent work in this direction is by Hacking \cite{Hac08}: he proves infinitesimal rigidity for the moduli stack $\overline{\mathcal{M}}_{g,n}$ of stable curves of arithmetic genus $g$ with $n$ marked points (a result that has been extended to positive characteristic in \cite{FM17}), and leaves the infinitesimal rigidity of $\mathcal{M}_{g,n}$ as an open question.

The rigidity of $\mathcal{A}_g$ is also open. To best of our knowledge, just the following is known. Let $\mathcal{A}'_g$ be a finite cover of $\mathcal{A}_g$ constructed as the moduli space of ppav with a level srtucture, we choose a level structure such that there is no difference between the coarse moduli space and the stack. Let $\overline{\mathcal{A}'}_g$ be a good toroidal compactification of $\mathcal{A}'_g$. Then, building on the classical work of Calabi and Vesentini \cite{CV}, in \cite{Pet18}*{Theorem 4.3} it is shown that the pair $(\overline{\mathcal{A}'}_g, \partial \overline{\mathcal{A}'}_g) $ is rigid. From an arithmetic point of view, similar rigidity results are proven in \cite{Fal84}.

%The problem of rigidity for the Torelli morphism has been addressed before this paper. 
%In \cite{Far21}, Farb can prove the global rigidity for the Torelli morphism for $g\geq 3$. Our contribution is to establish its infinitesimal rigidity.%, which, as previously observed, does not directly follow from global rigidity. Together with Farb's result, this provides a complete answer to the rigidity of the Torelli morphism. 

 \medskip
 
Our second result concerns infinitesimal rigidity of the Prym morphism $\PrymM:\cR_g\to \cA_{g-1}$, from the moduli stack $\cR_g$ of pairs $(C,\eta)$ where $C$ is a smooth projective curve of genus $g$ and $\eta\in J(C)$ is a non-trivial line bundle on $C$ with $\eta^{\otimes 2}\cong\cO_C$.  This morphism maps each pair $(C,\eta)$ to its Prym variety $\Pr(C,\eta)$ (see Section \ref{sec:proof2} for some details). It is never an immersion but it is generically injective for $g\geq 7$, namely as soon as the dimension of the target is larger than the dimension of the domain (see e.g. \cite{Don92}).

%Much more is known about this map: the most important results are collected in \cite{Don92}. 

As for $\mathcal{M}_g$,  any rigidity for $\cR_g$ is still unknown. However, the global rigidity of the Prym morphism is established in \cite{Ser22} and answers a question posed in \cite{Far21}. Again, our result of infinitesimal rigidity combined with the previous result on global rigidity provides a complete answer to the problem of rigidity with fixed target and domain. 
The following theorem is proven in Section \ref{sec:proof2}. 
 
\begin{theorem}\label{thm:main2} For any $g\geq 3$ the Prym morphism $\PrymM:\cR_g\to \cA_{g-1}$ does not admit any non-trivial first-order deformation with fixed domain and target. 
\end{theorem}

Our third result concerns the infinitesimal rigidity of the Spin-Torelli morphism $\SpinM:\cS_g\to \cN_g$, from the moduli stack $\cS_g$ of spin curves of genus $g$ to the moduli stack $\cN_g$ of pairs $(A,\Theta)$ of an abelian variety together with an effective symmetric divisor with $h^0(\cO_A(\Theta))=1$. Closed points in $\cS_g$ are pairs $(C,\vartheta)$ of a smooth projective curve $C$ of genus $g$ together with a theta-characteristic $\vartheta$ (i.e. a line bundle such that $\vartheta^{\otimes 2}\cong\omega_C$). The theta-characteristic allows to construct a unique \emph{symmetric} divisor $\Theta\subseteq J(C)$ and therefore a unique closed point in $\cN_g$ (see Section \ref{sec:proof3} for a more detailed description). This construction defines an injective morphism, the Spin-Torelli morphism. We are not aware of any result of rigidity regarding the moduli stacks $\cS_g$ %(\textcolor{red}{vero per $\cN_g$?}) 
or the Spin-Torelli morphism; the state of the art about the rigidity of $\cN_g$ is analogous to the case of $\cA_g$.

%\textcolor{red}{referenza su $S_g$ e $s\tau$}
The proof of the following theorem is contained in Section \ref{sec:proof3}.

\begin{theorem}\label{thm:main3} For any $g\geq 3$ the Spin-Torelli morphism $\SpinM:\cS_g\to \cN_g$ does not admit any non-trivial first-order deformation with fixed domain and target. 
\end{theorem}

We now focus on morphisms between moduli stacks of curves. By a result of Royden \cite{Roy71}, the only automorphism $\mathcal{M}_g\to \mathcal{M}_g$ is the identity. In \cite{Mas14}, building on \cite{GKM02}, it is shown that the automorphism group of $\overline{\mathcal{M}}_{g,n}$ is the symmetric group acting on the marked points, except for some low genera cases explicitely described in loc. cit. These problems are also reviewed in \cite{Far09}*{Question 4.6}.

Our next and last result regards infinitesimal rigidity of certain morphisms from $\mathcal{M}_g$ to another moduli stack of curves $\mathcal{M}_h$ of some genus $h\geq g$ constructed as follows.

Let $X_g$ (resp. $X_h$) be a closed orientable real surface of genus $g$ (resp. $h$). An unramified finite covering $p: X_h\to X_g$ gives a map  $p^*\colon \mathcal{T}_g\to \mathcal{T}_h$ between the corresponding Teichm\"{u}ller spaces by pulling back the complex structures. The cover $p$ is called {\em characteristic} if  $p_*(\pi_1(X_h))$ is a characteristic subgroup of $\pi_1(X_g)$, i.e. $p_*(\pi_1(X_h))$ is left invariant by $ \mbox{Aut}(\pi_1(X_g))$. Topologically, these are coverings such that every homeomorphism of $X_g$ lifts to a homeomorphism of $X_h$, and the lifting process defines a homomorphism $L_p : \mbox{Aut}(\pi_1(X_g)) \to \mbox{Aut}(\pi_1(X_h))$. Because of this, the map $p^*\colon \mathcal{T}_g\to \mathcal{T}_h$ defined by a characteristic cover descends to a morphism $p^*: \mathcal{M}_g\to \mathcal{M}_{h}$ (see \cite{BN97}*{III.1 and III.2} for more details).

The study of the global rigidity of these morphisms is stated as an open question in \cite{Far23} (see Question 4.5, attributed to C. McMullen). All these problems extend to the Deligne-Mumford compactification of $\mathcal{M}_g$ given the studies on the augmented Teichmüller space (see \cite{BN97} and \cite{HMQ21} for details).

Our contribution is to prove infinitesimal rigidity of all the morphisms $p^*:\mathcal{M}_g\to \mathcal{M}_{h}$ induced by characteristic covers $p\colon X_h\to X_g$.

\begin{theorem} \label{thm:main4}
For any $g\geq 3$, the morphism $p^*\colon\cM_g\to\cM_h$ does not admit any non-trivial first-order deformation with fixed domain and target.
\end{theorem}

It is natural to ask whether the compositions $\cM_g\to\cA_h$ of the morphisms $p^*\colon\cM_g\to\cM_h$ with the Torelli morphisms $\tau\colon\cM_h\to\cA_h$ are also rigid. The question on global rigidity is raised in \cite{Far23}. With our techniques, we cannot solve the corresponding question on infinitesimal rigidity for the moment. The main obstacle is discussed after the proof of Theorem \ref{thm:main4}.

\medskip
Let us stress that our results are for morphisms between stacks. The situation for the induced morphism between coarse spaces is discussed in Section \ref{sec:coarse}.

%\textcolor{blue}{Let us remark that our results as we prove them hold in general only in the category of stacks, where the tangent sheaves are torsion free. Being this an important ingredient of our proofs, it is not possible to conclude from them the analogue result on the underlying coarse moduli spaces. In section \ref{sec:coarse} we will explain the matter in detail.}

\medskip

Let us briefly explain the structure of the proofs. We first study the space of sections of the pullback of the tangent bundle of the target via the investigated morphism. For both smooth Deligne-Mumford stacks and normal varieties, the vanishing of all these sections suffices to conclude infinitesimal rigidity (see Lemma \ref{lem:def}). To prove this vanishing, we study those sections restricted to a complete curve $B$ in $\mathcal{M}_g$ through a general point of $\mathcal{M}_g$ and with a general tangent direction (see Lemma \ref{lem:KU} for details). We then conclude by relating these sections to the Hodge bundle associated to the family of genus $g$ curves over $B$, and using the positivity properties provided by the associated Fujita decomposition (\cite{Fuj78b}, \cite{CD:Answer_2017}). We explain this in detail in Subsection \ref{subsec:fujita}. Notice that such a complete curve $B$ exists for $g\geq 3$ because $\mathcal{M}_g$ admits a compactification with boundary of codimension $2$, the Satake compactification (see for instance \cite{Oor74}).% and such compactification is obtained by extending the Torelli morphism over the Deligne-Mumford compactification and mapping to the Satake compactification of $\mathcal{A}_g$. 

In genus $g=2$, the moduli space of curves is affine, hence non-trivial sheaves on $\cM_2$ have plenty of sections. This indicates that the above results should not hold in this case.

\medskip

We conclude the introduction with a couple of words about infinitesimal rigidity for moduli spaces of surfaces. On the one hand, moduli spaces of surfaces might deform \cite{Hac08}*{Section 6}. On the other hand, for higher dimensional varieties, the Torelli map generalizes to period maps. If we consider a rigid surface $S$ with strictly positive geometric genus or irregularity, then the domain of the period map is a point (the local deformation space of $S$) but the codomain has positive dimension (see, for instance \cite{CMP} Sections 4.4 and 4.5, Example 4.4.5). Hence, in this case the period map admits non-trivial deformations. Examples of these surfaces are the BCD surfaces constructed by Bauer and Catanese \cites{BC08,BC16} or the surfaces with $p_g=q=2$ constructed by Polizzi, Rito, and Roulleau \cite{PRR20}. 
%and of course, one can generalize this example to a higher dimension.
We do not know under which hypotheses the period maps of higher dimensional manifolds with positive dimensional moduli are rigid.

\section{Preliminaries}\label{sec:pre}

We work over the field of complex numbers.

\subsection{First-order deformations of morphisms}\label{subsec:firstorder}

In this paper, we are concerned about first order infinitesimal deformations of certain morphisms of stacks or normal varieties $f\colon X\to Y$ with fixed source and target. If $Y$ is a smooth variety, it is known that these deformations are classified by the global sections of $f^*\cT_Y$ \cite{Sernesi}*{Proposition 3.4.2, page 158}. This fact also holds in quite more general settings. Since we have not been able to find in the literature the statement in the generality we need, we include here a short proof (the same proof works in greater generality, but we give the statement only for our set-up). %We give the statement for smooth Deligne-Mumford stack, which is the case we need, but we think it holds even for non-reduced Artin stacks.

\begin{lemma} \label{lem:def}
Let $f\colon X\to Y$ be a morphism of either smooth Deligne-Mumford stacks or of normal varieties, and let $\cT_Y$ denote the tangent sheaf of $Y$. If $H^0(X, f^*\cT_Y)=0$, then all first-order deformations of $f$ with fixed source and target are trivial.
\end{lemma}
\begin{proof}
Let $D$ be the spectrum of $\mathbb{C}[\varepsilon]/(\varepsilon^2)$, and denote by $\{o\}$ its closed point. A first order deformation of $f\colon X\to Y$ is a $D$-morphism $\hat{f}\colon X\times D \to Y\times D$ which, when restricted to the central fibre $X_o:=X\times\{o\}\subset X\times D$ is equal to $f$.

Note that the ideal sheaf of the central fibre $X_o$ squares to zero, so that any $D$-morphism $X\times D\to Y\times D$ is uniquely determined by its restriction to $X_o$ and the restriction of its differential to $X_o$. Moreover, in the case of a first-order deformation $\hat{f}$, its restriction to the central fibre is given by $f$, hence one only needs to study $d\hat{f}_{|X_o}$.

Since it holds $\cT_{X\times D}=\cT_X\boxplus\cT_D$ and analogously $\cT_{Y\times D}=\cT_Y\boxplus\cT_D$, we have
$$\cT_X\boxplus\cT_D\cong \cT_{X\times D}\stackrel{d\hat{f}}{\longrightarrow}\hat{f}^*\cT_{Y\times D}\cong f^*\cT_Y\boxplus\cT_D,$$
hence we can describe $d\hat{f}_{|X_o}$ as
$$\left(
\begin{array}{c|c}
df & v \\
\hline 0 & 1
\end{array}
\right),$$
where $v:=(d\hat{f})\left(\frac{d}{d\varepsilon }\right)_{|X_o}$ and the $1$ in the low-right corner follows from $\hat{f}$ being a $D$-morphism.

From the hypothesis $H^0(X, f^*\cT_Y)=0$ follows that $v=0$. Thus $\hat{f}$ has the same differential as $f\times\id_D: X\times D\to Y\times D$, and by the above remark it follows $\hat{f}=f\times\id_D$ is the trivial deformation of $f$.

\begin{comment}
Using the differential of $\hat{f}$, we can construct the section $v:=(d\hat{f})\left(\frac{d}{d\varepsilon }\right)$ of $\hat{f}^*\cT_{\cA_g/D}\cong \tau^*\cT_{\cA_g} $. Let us show that if $v=0$, then $\hat{\tau}$ is trivial, i.e. it is equal to the Torelli map times the identity. To this end, observe that both maps are equal when restricted to $\left(\cM_g\times D\right)_{red}=\cM_g$, and have the same differential. This is enough to conclude because ideal sheaf of nilpotent functions on $\cM_g\times D$ squares to zero.

Let $D$ be the spectrum of $k[\varepsilon]/(\varepsilon^2)$. A first order deformation of $\tau$ is a morphism $\hat{\tau}\colon \cM_g\times D \to \cA_g\times D$ relative to $D$ which, when restricted to the fibre over the closed point $o\in D$, is equal to the Torelli map. Using the differential of $\hat{\tau}$, we can construct the section $v:=(d\hat{\tau})\left(\frac{\partial}{\partial\varepsilon }\right)_{|\cM_g}$ of $\left(\hat{\tau}^*\cT_{\cA_g\times D/D}\right)_{|\cM_g}\cong \tau^*\cT_{\cA_g}$, where we identify $\cM_g$ with $\cM_g\times\{o\}$, the fibre over the closed point of $D$.

Let us show that if $v=0$, then $\hat{\tau}$ is trivial, i.e. it is equal to the Torelli map times the identity. To this end, observe that both maps are equal when restricted to $\left(\cM_g\times D\right)_{red}=\cM_g$, and have the same differential. This is enough to conclude because ideal sheaf of nilpotent functions on $\cM_g\times D$ squares to zero.
\end{comment}
\end{proof}

%\begin{lemma}\label{lem:def} The space parametrizing the first-order deformations with fixed domain and fixed target of the Torelli morphism $\tau:\cM_g\to \cA_g$ is $H^0(\cM_g, \tau^*\cT_{\cA_g}).$
%\end{lemma}
%\begin{proof} 

%The proof of the lemma strongly relies on the analogous case of morphisms of smooth schemes (\cite{Ser}*{Proposition 3.4.2}), where the reader can easily check that the assumption of projectivity can be dropped from the statement. Now recalling that both $\cM_g$ and $\cA_g$ are smooth algebraic stack and considering a covering of charts in smooth schemes we deduce the result for stack from the one for scheme \rtext{(è vero che si può far per carte e incollare o dobbiamo guardare l'annullamento dell'$H^1$? sennò come si fa? Questa è per Giulio esperto di Stack!)}
%\end{proof}
%\begin{remark} The proof of the previous lemma do not use any other property of the Torelli morphism than the fact that it is a morphism of smooth algebraic stack. The a statement in fact general of smooth algebraic stack applied to our special case, but we couldn't find a precise reference of this in literature.

%\end{remark}
\subsection{Fujita decompositions on the Hodge bundle}\label{subsec:fujita} Let $f:S\to B$ be a fibration from a smooth projective surface $S$ to a smooth projective curve $B$, namely a family of projective curves of arithmetic genus $g$ over a smooth projective curve $B$ whose general fibre is smooth. Denote by $q_f$ the relative irregularity of $f$, defined as the difference $q_f=q(S)-q(B)$ of the irregularities of $S$ and $B$. To any such $f$ one can associate a Hodge bundle $f_\ast\omega_{S/B}=f_\ast\omega_S\otimes \omega_B^\vee$ whose general fibre is of rank $g$ isomorphic to $H^0(C_b, \omega_{C_b})$, where $C_b=f^{-1}(b)$.
\begin{theorem}[\cite{Fuj78b}, \cite{CD:Answer_2017}]\label{Fujita}
The Hodge bundle $f_\ast\omega_{S/B}$ has decompositions of vector bundles 
\begin{equation}
    f_\ast\omega_{S/B}=\cO_{B}^{q_f}\oplus \cV=\cU\oplus \cA
\end{equation}
where $\cA$ is ample and $\cU$ is unitary flat, which are compatible in the sense that $\cO_B^{q_f}\subset \cU$ as vector bundle provides a splitting $\cU=\cO_B^{q_f}\oplus \cU'$.
\end{theorem}

Fujita decompositions are strongly related to the infinitesimal variation of the Hodge structure (ivhs in short), namely with the coboundary morphism $\theta_b: H^0(C_b, \omega_b)\to H^1(C_b, \cO_{C_b})$ of the short exact sequence attached to the first order deformation $\xi_b\in Ext^1(\omega_{C_b},\cO_{C_b})\cong  H^1(T_{C_b})$ induced by $f$ on the fibre $C_b$. Suppose that $f$ is semistable, namely that the relative canonical bundle is $f$-ample and the singular fibers are reduced with at most nodal singularities; then if $\Gamma\subset B$ denotes the set of critical values and $\Upsilon=f^*\Gamma$, there is a canonical isomorphism $f_\ast\omega_{S/B}\simeq f_\ast\Omega^1_{\cC/B}( \log \Upsilon)$, where the latter bundle is defined by the short exact sequence
$$0\to f^\ast \omega_B (\log \Gamma) \to \Omega^1_\cC( \log \Upsilon) \to \Omega^1_{\cC/B}(\log \Upsilon) \to 0.$$
The  connecting homomorphism
\begin{equation} \label{eq:theta}
\theta: f_\ast\omega_{S/B}\simeq f_\ast \Omega^1_{\cC/B}\left(\log\Upsilon\right)\longrightarrow  R^1f_\ast \cO_{\cC}\otimes \omega_B(\log \Gamma)
\end{equation}
is a morphism of locally free sheaves which on the fibres over $b\not\in\Gamma$ coincides with $\theta_b$. The kernel $\cK=\ker \theta$ is a vector subbundle of $f_\ast\omega_{S/B}$ whose fibre over a general $b\in B$ is $\ker\theta_b$. There are natural inclusions $\cU\subseteq\cK\subseteq f_*\omega_{S/B}$. 

We refer to \cite{GT20} for more details on the last paragraph, and a treatment of the non-semistable case, which requires more care and it is not used in this note.

%$f$ is not semistable, the Fujita decomposition still exists and thus $\cU$ can be defined, but there is not automatically a global gluing for $\theta_b$ to $\theta$ and so $\cK$ might be defined only after semistable reduction, which for monodromy reasons can make $\cU$ larger. 

\begin{lemma}\label{lem:KU} Let $\overline{\cM}_g$ be the moduli stack of stable curves of genus $g$. A general complete curve $\pi \colon B\to \overline{\cM}_g$ corresponds to a semistable fibration with $\cU=0$ (more precisely, there exists an open dense subset $U$ of the tangent bundle $\cT\overline{\cM}_g$ such that if the image of $d\pi$ intersects $U$, then $\cU=0$).
\end{lemma}
\begin{proof}
Curves in $\overline{\cM}_g$ correspond by construction to semistable fibrations. For a general smooth curve $[C_b]$ in $\cM_g$ and a general direction $\xi_b\in T_{[C_b]}\overline{\cM}_g \simeq H^1(C_b,T_{C_b})\simeq Ext^1(\omega_{C_b}, \cO_{C_b})$, the induced linear map $\theta_b\colon H^0(C_b,\omega_{C_b})\to H^1(C_b,\cO_{C_b})$ has maximal rank (see for example \cite{GT20}, \cite{LP16}*{Lemma 2.4}), so the fibre $\cK_b$ is zero. As $\cU$ is locally free and contained in $\cK$, we obtain the statement. %Since $\cU$ is unitary flat we conclude that $\cU=0.$
\end{proof}

\subsection{Ample vector bundles on curves}

\begin{lemma} \label{lem:no-sections}
    If $\cA$ is an ample vector bundle over a smooth projective curve $B$, then it holds $H^0(B,\Sim^n\cA^{\vee})=0$ for every $n>0$.
\end{lemma}
\begin{proof}
    Note first that if $\cA$ is ample, then so is $\Sim^n\cA$ (\cite{LazII}*{Theorem 6.1.15}). In particular any quotient line bundle $\Sim^n\cA\twoheadrightarrow Q$ is ample on $B$ (\cite{LazII}*{Proposition 6.1.2}), i.e. $\deg Q>0$.
    
    Suppose $H^0\left(\Sim^n\cA^{\vee}\right)\neq 0$ and let $\sigma$ be a non-zero section, which induces a morphism of sheaves $\sigma\colon\cO_B\to\Sim^n\cA^{\vee}$. Dualizing it we obtain a non-zero map $\Sim^n\cA\to\cO_B$, whose image is a quotient of $\Sim^n\cA$ and a non-zero subsheaf $Q\subseteq\cO_B$. In particular $Q$ is torsion-free, and hence a locally-free sheaf because $B$ is a smooth curve. Moreover $\deg Q\leq\deg\cO_B=0$, contradicting the amplitude of $\Sim^n\cA$.
\end{proof}

\section{Proof of Theorem \ref{thm:main1}}\label{sec:proof1}
In this section we give the proof of Theorem \ref{thm:main1}, asserting that the Torelli morphism $\tau: \cM_g\to \cA_g$ does not admit non-trivial first-order deformations. 

Recall that $\cM_g$ denotes the moduli stack of smooth projective curves of genus $g$, $\cA_g$ the moduli stack of principally polarized abelian varieties of dimension $g$, and $\TorM:\cM_g\to\cA_g$ the Torelli morphism, which at the level of points maps (the isomorphism class of) a smooth projective curve $C$ to its Jacobian variety $J(C)\cong H^1\left(C,\cO_C\right)/H^1\left(C,\mathbb{Z}\right)$ with its natural principal polarization $\Theta_C$.

The tangent space to $\cM_g$ at $\left[C\right]$ is $H^1\left(C,T_C\right)$, and the tangent space to $\cA_g$ at $\left[J\left(C\right),\Theta_C\right]$ is
$$\Sym^2H^1\left(C,\cO_C\right)\cong\Sym^2 H^0\left(C,\omega_C\right)^{\vee}\cong\Hom^s\left(H^0\left(C,\omega_C\right),H^1\left(C,\cO_C\right)\right),$$
where $\Hom^s$ denotes the set of symmetric (i.e. self-dual) linear maps.

Moreover the image of $\xi\in H^1\left(C,T_C\right)\cong T_{[C]}\cM_g$ under the differential of $\TorM$ can be identified (up to non-zero scalar) with the multiplication map (cup-product followed by contraction)
$$H^0\left(C,\omega_C\right)\to H^1\left(C,\cO_C\right),\quad \alpha\mapsto \xi\cdot\alpha$$

By Lemma \ref{lem:def}% of Section \ref{subsec:firstorder}
, Theorem \ref{thm:main1} follows from the following vanishing:
\begin{theorem}
If $g \geq 3$, then $H^0(\cM_g, \tau^*\cT_{\cA_g})=0$.
\end{theorem}

\begin{proof}
Since $\tau^*\cT_{\cA_g}$ is locally free and $\cM_g$ is reduced, if we show that for every point in a dense subset  of $\cM_g$ there exists a curve $\pi \colon B \to \cM_g$ such that $h^0(B,\pi^*\tau^*\cT_{\cA_g})=0$, then $h^0(\cM_g,\tau^*\cT_{\cA_g})=0$.

When $g \geq 3$, the coarse moduli space of $\cM_g$ admits a normal projective compactification whose boundary has codimension two, namely the Satake compactification. In $\cM_g$, we can look at the open subset $M_g^0$ of curves with trivial automorphism group, which is representated by a smooth scheme and whose complement in the Satake compactification has also codimension two. Because of this, for every point $[C]$ of $M_g^0$ and every tangent direction $v$ in $\cT_{[C]}\cM_g$, we can find a smooth projective curve $\pi\colon B \to \cM_g$ passing trough $[C]$ and tangent to $v$.

 Consider such a curve $B$, the corresponding family of curves $f\colon\cC\to B$ and the Hodge bundle $E=f_*\Omega^1_{\cC/B}$, whose fibre over a point $\left[C\right]\in B$ is $H^0\left(C,\omega_C\right)$. By the above discussion, the restriction of $\pi^*\tau^*\cT_{\cA_g}$ to $B$ is naturally isomorphic to $\Sim^2E^\vee$. Now by Theorem \ref{Fujita}, $E$ carries a decomposition $E=\cA\oplus \cU$ with $\cA$ ample and $\cU$ unitary flat, and by Lemma \ref{lem:KU}, $\cU$ is zero for general $B$ and $v$. By Lemma \ref{lem:no-sections}, $h^0(B,\pi^*\tau^*\cT_{\cA_g})=0$.

\end{proof}

\section{Proof of Theorem \ref{thm:main2}}\label{sec:proof2}

As already recalled in the introduction, we denote by $\cR_g$ the moduli stack of pairs $\left(C,\eta\right)$, where $C$ is a smooth projective curve of genus $g$, and $\eta\in J\left(C\right)$ is a non-trivial line bundle of order two (i.e. $\eta^{\otimes2}\cong\cO_C$). By standard theory, such a pair is equivalent to an étale double cover $\pi\colon C'\to C$ where $C'$ is a connected smooth projective curve and
\begin{equation} \label{eq:Prym-cover}
\pi_*\cO_{C'}=\cO_C\oplus\eta.
\end{equation}
More precisely, since $\pi$ is finite, there is a trace morphism $\Tr\colon\pi_*\cO_{C'}\to\cO_C$ that splits the structure morphism $\cO_C\to\pi_*\cO_{C'}$, and $\eta=\ker\Tr$.

One way to define the Prym variety $\PrymM\left(C,\eta\right)$ of the pair $\left(C,\eta\right)$ (or the cover $\pi\colon C'\to C$) is as the cokernel of the pull-back map
$$\pi^*\colon J\left(C\right)\to J\left(C'\right),$$
which has dimension $\dim\PrymM\left(C,\eta\right)=g\left(C'\right)-g\left(C\right)=g-1$ (note that $g\left(C'\right)=2g\left(C\right)-1$ by the Hurwitz formula).

Alternatively $\PrymM\left(C,\eta\right)$ can be defined as the connected component trough $\left[\cO_{C'}\right]$ of the Norm map $\Nm\colon J\left(C'\right)\to J\left(C\right)$.

The natural principal polarization of $J\left(C'\right)$ induces twice a principal polarization $\Xi$ on $\PrymM\left(C,\eta\right)$. The Prym morphism $\PrymM\colon\cR_g\to\cA_{g-1}$ is then defined (at the level of $\mathbb{C}$-points) as
$$\left[C,\eta\right]\to \left[\PrymM\left(C,\eta\right),\Xi\right].$$

Since $\cR_g$ is an étale cover of $\cM_g$ (of degree $2^{2g}-1$), the tangent space of $\cR_g$ at a point $\left[C,\eta\right]$ is naturally isomorphic to $T_{[C]}\cM_g\cong H^1\left(C,T_C\right)$.

On the other hand, the tangent spaces of the Jacobians $J\left(C\right)$ and $J\left(C'\right)$ (at the corresponding neutral elements) are $H^1\left(C,\cO_C\right)$ and $H^1\left(C',\cO_{C'}\right)$. Thus the tangent space of $\PrymM\left(C,\eta\right)$ is naturally isomorphic to
$$H^1\left(C',\cO_{C'}\right)/\pi^*H^1\left(C,\cO_C\right)\cong H^1\left(C,\eta\right).$$
In the last isomorphism we have combined equation \eqref{eq:Prym-cover} and the fact that $\pi$ is a finite morphism to obtain
$$H^1\left(C',\cO_{C'}\right)\cong H^1\left(C,\pi_*\cO_{C'}\right)\cong H^1\left(C,\cO_C\right)\oplus H^1\left(C,\eta\right).$$

Therefore the tangent space of $\cA_{g-1}$ at $\PrymM\left(C,\eta\right)$ can be naturally identified with
\begin{align*}
\Sym^2T_0\PrymM\left(C,\eta\right)&\cong\Sym^2H^1\left(C,\eta\right)\cong\Hom^s\left(H^1\left(C,\eta\right)^{\vee},H^1\left(C,\eta\right)\right)\\
&\cong\Hom^s\left(H^0\left(C,\omega_C\otimes\eta\right),H^1\left(C,\eta\right)\right),
\end{align*}
where the last isomorphism follows from Serre duality and $\eta^{\otimes 2}\cong\cO_C$.

Finally the differential of the Prym morphism at $\left[C,\eta\right]$,
$$d\PrymM_{\left[C,\eta\right]}\colon H^1\left(C,T_C\right)\to \Sym^2H^1\left(C,\eta\right)\cong\Hom^s\left(H^0\left(C,\omega_C\otimes\eta\right),H^1\left(C,\eta\right)\right)$$
is induced by cup-product (up to a non-zero scalar).%, i.e. $d\PrymM_{\left[C,\eta\right]}\left(\xi\right)=\xi\cdot$ for any $\xi\in H^1\left(C,T_C\right)$.

As in the above section, Theorem \ref{thm:main2} follows from Lemma \ref{lem:def} and the following vanishing:
\begin{theorem}
When $g \geq 3$, it holds $H^0(\cR_g, \PrymM^*\cT_{\cA_{g-1}})=0$.
\end{theorem}

\begin{proof}
By construction, for a  given curve $C$ there are $2^{2g}-1$ choices of $\eta$, and indeed this gives a natural étale morphism $\varphi\colon\cR_g\to\cM_g$ of degree $2^{2g}-1$. Set $R^0_g=\varphi^{-1}\left(M_g^0\right)$, local chart of $\cR_g$ corresponding to coverings $C'\to C$ where $C$ has trivial autormorphism group. Moreover, since $M^0_g$ can be covered by smooth projective curves, so does $R^0_g$ by taking the connected components of the preimages under $\varphi$.

Let now $B\subseteq R_g^0$ be a general smooth curve, which corresponds to a family of coverings $f'\colon\cC'\stackrel{\pi}{\to}\cC\stackrel{f}{\to} B$. The induced morphism $\pi$ is also a étale double cover of surfaces. The trace of $\pi$ gives a splitting $\pi_*\cO_{\cC'}\cong\cO_{\cC}\oplus\cL$, where $\cL=\ker\Tr$ restricts to $\eta$ on a fibre $C'\to C$. In particular we also have \begin{equation} \label{eq:splitting-Prym}
R^1f'_*\cO_{\cC'}\cong R^1f_*\cO_{\cC}\oplus R^1f_*\cL,
\end{equation}
and by the above discussion on tangent spaces, there is a natural identification
$$\PrymM^*\cT_{\cA_{g-1}}\cong\Sym^2\left(R^1f'_*\cO_{\cC'}/R^1f_*\cO_{\cC}\right)\cong\Sym^2R^1f_*\cL.$$
By relative duality, the equation \eqref{eq:splitting-Prym} gives
$$f'_*\Omega^1_{\cC'/B}\cong f_*\Omega^1_{\cC/B}\oplus\left(R^1f_*\cL\right)^{\vee}$$
so that $\left(R^1f_*\cL\right)^{\vee}$ is isomorphic to a quotient of the Hodge bundle $E'=f'_ *\Omega^1_{\cC'/B}$ of $f'$.

We can now adapt the proof of Lemma \ref{lem:KU} to show that $E'$ is ample for a general $B$, and then Lemma \ref{lem:no-sections} concludes the proof of the theorem.

Let $C'\to C$ be one of the coverings of the family $\pi$, with corresponding $\eta\in\Pic^0\left(C\right)$, and let $\xi\in H^1\left(C,T_C\right)\cong T_{[C]}M_g^0\cong T_{[C'\to C]}R_g^0$ be a tangent vector to $B$. Using the decompositions
$$H^0\left(C',\omega_C'\right)=H^0\left(C,\omega_C\right)\oplus H^0\left(C,\omega_C\otimes\eta\right)\quad\text{and}\quad H^1\left(C',\cO_{C'}\right)=H^1\left(C,\cO_C\right)\oplus H^1\left(C,\eta\right)$$
the infinitesimal variation of Hodge structure $\cC'\to B$ at this point is ``diagonal'', given by multiplication with $\eta$ on each component. By \cite{LP16}*{Lemma 2.4}, a general $\xi\in H^1\left(C,T_C\right)$ gives isomorphisms in both components. Taking a smooth projective curve in $M_g^0$ through $\left[C\right]$ with tangent vector $\xi$, and then $B$ the appropriate connected component of its preimage in $R^g_0$ induces a family $f'\colon\cC'\to B$ with ample Hodge bundle, as wanted.
\end{proof}

\section{Proof of Theorem \ref{thm:main3}}\label{sec:proof3}

Recall from the introduction that $\cS_g$ is the moduli stack of pairs $\left(C,\vartheta\right)$ of projective curves of genus $g$ with a theta characteristic $\vartheta$, i.e. $\vartheta\in\Pic^{g-1}\left(C\right)$ such that $\vartheta^{\otimes 2}\cong\omega_C$. Since two theta characteristics differ by a $2$-torsion element of the $g$-dimensional abelian variety $\Pic^0\left(C\right)$, the natural forgetful morphism $\pi\colon\cS_g\to\cM_g$ defined on points by $\left(C,\vartheta\right)\to C$ is étale of degree $2^{2g}$.

On the other side, $\cN_g$ denotes the moduli stack of pairs $\left(A,\Theta\right)$, where $A$ is an abelian variety of dimension $g$ and $\Theta\subseteq A$ is a symmetric divisor inducing a principal polariaztion on $A$, i.e. $-\Theta=\Theta$ and $h^0\left(\cO_A\left(\Theta\right)\right)=1$.

In order to describe the morphism $\SpinM\colon\cS_g\to\cN_g$, let's first quickly recall one construction of the principal polarization on the jacobian variety $J\left(C\right)=\Pic^0\left(C\right)$ of a projective curve $C$ of genus $g$. There is a natural morphism
$$\varphi\colon C^{g-1}\to\Pic^{g-1}\left(C\right),\quad \left(p_1,\ldots,p_{g-1}\right)\mapsto\left[\cO_C\left(p_1+\ldots+p_{g-1}\right)\right].$$
By the Riemann parametrization theorem, $\varphi$ is birational onto its image, which is thus a divisor. Moreover, its image is precisely the set $W_{g-1}^0=\left\{L\in\Pic^{g-1}\left(C\right)\,|\,h^0\left(L\right)>0\right\}$. Any fixed $\vartheta\in\Pic^{g-1}\left(C\right)$ induces an isomorphism (of algebraic varieties)
$$\Pic^{g-1}\left(C\right)\to\Pic^0\left(C\right)=J\left(C\right),\quad\left[L\right]\mapsto \left[L\right]-\left[\vartheta\right]:=\left[L\otimes\vartheta^{\vee}\right].$$
The image $\Theta_{\vartheta}:=W_{g-1}^0-\left[\vartheta\right]$ of $W_{g-1}^0$ is thus a divisor in $J\left(C\right)$ and induces the principal polarization used in the Torelli morphism $\tau$. An easy application of Riemann-Roch shows that $\Theta_{\vartheta}$ is symmetric if and only if $\vartheta^{\otimes 2}\cong\omega_C$.

The morphism $\SpinM$ is defined by mapping $\left[S,\vartheta\right]$ to the pair $\left(J\left(C\right),\Theta_{\vartheta}\right)$.

The above discussion also shows that the natural principal polarization on $J\left(C\right)$ can be represented by exactly $2^{2g}$ symmetric divisors, hence the forgetful morphism $\pi'\colon\cN_g\to\cA_g$ is also étale of degree $2^{2g}$.

As in the two preceding cases, Theorem \ref{thm:main3} follows from Lemma \ref{lem:def} and the following vanishing:
\begin{theorem}
When $g\geq 3$, it holds $H^0(\cS_g, \SpinM^*\cT_{\cN_g})=0$.
\end{theorem}

\begin{proof}
As in the previous two proofs, it is enough to show that a general point of $\cS_g$ is contained in a projective curve $B\subseteq\cS_g$ such that $H^0\left(B,\SpinM^*\cT_{\cN_g}\right)=0$.

To this aim consider the natural commutative diagramm with étale vertical arrows
$$\xymatrix{
\cS_g \ar[rr]^{\SpinM} \ar[d]_{\pi} & & \cN_g \ar[d]^{\pi'} \\ \cM_g \ar[rr]^{\tau} & & \cA_g
}$$
Note that $\cT_{\cN_g}=\left(\pi'\right)^*\cT_{\cA_g}$ because $\pi'$ is étale, and thus $\SpinM^*\cT_{\cN_g}=\pi^*\tau^*\cT_{\cA_g}$.

By the proof of Theorem \ref{thm:main1}, the general point of $M_g^0$ is contained in a smooth projective curve $B'\subseteq M_g^0$ such that the Hodge bundle $E'=\left(f'\right)_*\Omega^1_{\cC'/B'}$ of the corresponding family of curves $f'\colon\cC'\to B'$ is ample. Any connected component $B\subseteq\pi^{-1}\left(B'\right)$ corresponds to a family of smooth projective curves $f\colon\cC\to B$ (with a family of theta characteristics), which is nothing but the pull-back of $f'$ by the étale Morphism $\varphi=\pi_{|B}\colon B'\to B$. In particular, the Hodge bundle of $f$
$$E:=f_*\Omega^1_{\cC/B}\cong\varphi^*E'$$
is also ample.

Thus a general point of $S_g^0\:=\pi^{-1}\left(M_g^0\right)$ is contained in a smooth projective curve $B\subseteq S_g^0$ such that $\left(\pi^*\tau^*\cA_g\right)_{|B}\cong \Sim^2E^\vee$ with ample $E$. Lemma \ref{lem:no-sections} implies that $H^0\left(B,\pi^*\tau^*\cA_g\right)=0$, as wanted.
\end{proof}

%\subsection{Super Torelli morphism}
\begin{remark}[Super Torelli morphism] \label{rmk:Super-Torelli}
It is possible also to define a period map for the moduli space of Supersymmetric Riemann surfaces. Its target is again $\cN_g$. As explained in \cite{super}, this map is rational and factors trough a non-reduced classical stack $M$ ($=\mathfrak{M}^+_g/\Gamma$, in the notations of loc.cit.). The reduced stack underlying $M$ is the irreducible component $\cS_g^+$ of $\cS_g$ where the spin structure has an even number of sections. The restriction of the period map to $\cS_g^+$ is the Spin-Torelli map studied also in this note. We do not know if this generalization of the Spin-Torelli map is rigid.
\end{remark}

\section{Proof of Theorem \ref{thm:main4}}\label{sec:proof4}

Let $X$ be a closed orientable real surface of genus $g$. An unramified finite covering $p: X'\to X$ is called {\em characteristic} if it corresponds to a characteristic subgroup of the fundamental group $\pi_1(X),$ namely $\pi_1(X')$ as a subgroup of $\pi_1(X)$ must be left invariant by every element of $ \mbox{Aut}(\pi_1(X))$. Topologically, these are coverings such that every homeomorphism of $X$ lifts to a homeomorphism of $X'$ and the lifting process defines a homomorphism $L_p : \mbox{Aut}(\pi_1(X)) \to \mbox{Aut}(\pi_1(X'))$. 

The moduli $\mathcal{M}_i$ of curves of genus $i$ is realized as the quotient of the Teichmüller space $\mathcal{T}_i$ by the mapping class modular group $DM_i$. Any characteristic cover $p\colon X'\to X$ defines a map $\mathcal{T}_g\to \mathcal{T}_h$ (where $g=g(X)$ and $h=g(X')$). By using $L_p$, such a morphism descends to a morphism $p^*: \mathcal{M}_g\to \mathcal{M}_{h}$ (see \cite{BN97}*{III.1 and III.2} for more details).

%As mentioned in the introduction, let $p\colon S_h\to S_g$ be a given characteristic topological covering between the compact orientable real surfaces of genus $h$ and $g$, and let $\tau_p:=p^*\colon\cM_g\to\cM_h$ be the morphism obtained by pulling back to $S_h$ the complex structure on $S_g$, so that $p$ becomes a holomorphic map. 

The statement of Theorem \ref{thm:main4} follows now from Lemma \ref{lem:def} and the following theorem.

\begin{theorem}
When $g \geq 3$, it holds $H^0(\cM_g, p^*\cT_{\cM_h})=0$.
\end{theorem}
\begin{proof}
    We use the same strategy as in the previous proofs. A smooth projective curve $B\subseteq\cM_g$ corresponds to a non-isotrivial family $\pi\colon\cC\to B$ of smooth projective curves of genus $g$, and the morphism $\tau_p\colon\cM_g\to\cM_h$ produces a non-isotrivial family $\pi'\colon\cC'\to B$ of curves of genus $h$ such that over $b\in B$ there is covering $C'_b\to C_b$ induced by $p$. The fibre of $p^*\cT_{\cM_h}$ over $b$ is $H^1\left(C'_b,T_{C'_b}\right)$ (the tangent space to $\cM_h$ at $C'_b=\tau_p\left(C_b\right)$), and hence
    $$p^*\cT_{\cM_h\mid B}=R^1\pi'_* T_{\cC'/B}\cong\left(\pi'_*\omega_{\cC'/B}^{\otimes 2}\right)^{\vee}.$$
    By \cite{EV90}*{Theorem 3.1}, the bundle $\pi'_*\omega_{\cC'/B}^{\otimes 2}$ is ample on $B$, hence by Lemma \ref{lem:no-sections} we have $H^0\left(B,p^*\cT_{\cM_h\mid B}\right)=0$. The proof finishes as in the previous cases, by noticing that $\cM_g$ can be covered by such smooth projective curves $B$ and $p^*\cT_{\cM_h}$ is torsion-free, so that $H^0\left(B,p^*\cT_{\cM_h}\right)=0$ for a general $B$ implies $H^0\left(\cM_g,p^*\cT_{\cM_h}\right)=0$ as wanted.
\end{proof}

Composing the above studied map $ p^*\colon\cM_g\to\cM_h$, with the Torelli map $\cM_h \to \cA_h$ we obtain a morphism $\cM_g\to \cA_h$. Our methods do not apply immediately to the study of its rigidity, one needs a more sophisticated understanding of the inclusion $p^*\left(\cM_g\right)\subseteq\cM_h$ and its tangent spaces. More precisely, to prove the relevant generalization of Lemma \ref{lem:KU}, given a étale covering $\pi\colon C_h\to C_g$, it would be necessary to understand if a general first order deformation of $C_h$ \emph{compatible with $\pi$} also satisfies the properties of \cite{LP16}*{Lemma 2.4}.

Given a possibly covering $\pi\colon C_h\to C_g$ as above, one could also consider the generalized Prym variety $\Pic^0\left(C_h\right)/\pi^*\Pic^0\left(C_g\right)$ (which inherits a polarization of a certain type $\delta$ depending on the topological type of $\pi$) and thus construct a generalized Prym morphism $\cM_g\to\cA_{h-g}^{\delta}$ from the moduli stack of curves to that of $(h-g)$-dimensional abelian varieties with polarization of type $\delta$. The study of its rigidity presents the same difficulties of that of $\cM_g\to\cA_h$ introduced above.

\section{Remarks about rigidity of coarse morphisms}\label{sec:coarse}
Given an (infinitesimally) rigid morphism of stacks $F\colon \mathcal{X} \to \mathcal{Y}$, one can ask if the corresponding morphism of coarse spaces $f\colon X \to Y$ is also (infinitesimally) rigid. The answer in this generality is negative, as the following example shows.
\begin{example}\label{ex:nonrigid}
Let $G$ be the group $\mathbb{Z}_2$, $\mathcal{X}=BG$ be the quotient stack of a point by $G$, and $\mathcal{Y}$ the quotient stack of the affine line by the action of $G$ which maps $x$ to $-x$. There is a unique morphism $F\colon\mathcal{X}\to\mathcal{Y}$, which maps $BG$ to the fixed point of the action and is infinitesimally rigid. However, the corresponding map of coarse spaces does deform. (To make contact with the forthcoming Proposition \ref{prop:stackVScoarse}, note that in this case $F^{-1}(U)$ is the empty set.)
\end{example}

It is natural to wonder if the coarse version of the modular morphisms considered in this paper are rigid. Concerning infinitesimal rigidity, taking into account the criterion given in Lemma \ref{lem:def}, we can ask under which conditions $H^0\left(\cX,F^*T_{\cY}\right)=0$ implies $H^0\left(X,f^*T_Y\right)=0$. 

The following definition will be useful. Let $V$ be a sheaf on a variety $X$, and $T(V)$ the torsion subsheaf of $V$. The inclusion $T(V)\hookrightarrow V$ induces an inclusion $i\colon H^0(X,T(V))\hookrightarrow H^0(X,V)$. We say that a global section of $V$ is a torsion section if it is in the image of $i$. We have the following partial result.

\begin{proposition}\label{prop:stackVScoarse}
Let $F\colon \mathcal{X} \to \mathcal{Y}$ be a morphism of smooth Deligne-Mumford stacks such that $H^0\left(\cX,F^*T_{\cY}\right)=0$. Let $p\colon \mathcal{X}\to X$ and $q\colon \mathcal{Y} \to Y$ be the maps to coarse spaces and $f\colon X \to Y$ the morphism induced by $F$. Let $U\subseteq\mathcal{Y}$ be the open subset where $q$ is \'etale, and assume that $F^{-1}(U)\subseteq\mathcal{X}$ is a big open subset (i.e. its complement has codimension at least $2$). Then all sections $H^0\left(X,f^*T_{Y}\right)$ are torsion sections. 
\end{proposition}
\begin{proof}
Over a field of characteristic zero, coarse moduli spaces of DM stacks are good, hence the adjunction morphism $f^*\cT_Y\to p_*p^*f^*\cT_Y$ is an isomorphism, and thus $H^0\left(X,f^*\cT_Y\right)=H^0\left(X,p_*p^*f^*\cT_Y\right)=H^0\left(\mathcal{X},p^*f^*\cT_Y\right)$. We have $p^*f^*\cT_Y=F^*q^*\cT_Y$. On $U$, we have $q^*\cT_Y=\cT_\mathcal{Y}$. Assume by contradiction that there exists an element $s$ of $H^0(\mathcal{X},p^*f^*\cT_Y)$ which is not torsion. Since it is not torsion, its restriction to $F^{-1}(U)$ is not zero, hence $H^0(F^{-1}(U),p^*f^*\cT_Y|_{F^{-1}(U)})\neq 0$. Then $H^0(F^{-1}(U), F^*\cT_{\cY}|_{F^{-1}(U)})\neq 0$.

As $F^{-1}(U)$ is big in $\mathcal{X}$ and $F^*\cT_{\mathcal{Y}}$ is locally free, we obtain that $H^0(\mathcal{X},F^*\cT_{\mathcal{Y}})\neq  0$, contradicting the hypothesis.
\end{proof}

Let us check whether the hypotheses of Proposition \ref{prop:stackVScoarse} are satisfied in our cases. On the one hand, the map from the moduli stack of ppav (and its variants discussed in this paper) to its coarse moduli space is \'{e}tale over the closed points with automorphism group exactly $\{\pm 1\}$. The preimage of this open set via the various Torelli, Spin and Prym maps is the open set of curves without automorphisms, which is big in the moduli space of smooth curves and its variants when $g\geq 4$, see e.g. \cite{Hac08}*{Lemma 5.3}. 

On the other hand, characteristic covers are Galois, so the image of $p^*$ lies in the locus of genus $h$ curves with non-trivial automorphisms. This means that Proposition \ref{prop:stackVScoarse} cannot be applied to the morphisms induced by characteristic covers.

Note that torsion sections of $f^*\cT_Y$ might exist even in simple cases, as the following example shows, so we cannot exclude that the Torelli morphism has infinitesimal deformations.

\begin{example}
Let $Y\subseteq\mathbb{C}^3$ be a quadratic cone, which has a normal singularity at the vertex. Take as $X$ a line trough the vertex, and $f$ the inclusion. Then $f^*\cT_Y$ is a rank two sheaf on $X\cong\mathbb{C}$ with torsion at the origin (%To see that there is torsion, observe that the stalk of $\cT_Y$ at the origin is three dimensional, so the stalk of $f^*\cT_Y$ at the origin is at least three dimensional, and since we are on a curve and the rank is two there is torsion. 
a direct computation reveals that the torsion subsheaf is a skyscraper sheaf with two-dimensional fiber).
\end{example}

Unfortunately, we do not know of any systematic study of infinitesimal deformations coming from torsion sections. Let us pose the following general question in deformation theory, whose study goes beyond the scope of this paper.
\begin{question}\label{q}
Let $f\colon X \to Y$ be a morphism of normal varieties, if all global sections of $f^*\cT_Y$ are torsion, is $f$ locally rigid?
\end{question}
A variant of the phenomenon encountered here is studied in \cite{AC}*{Section 6}; in the spirit of loc. cit., we might speculate that the answer to Question \ref{q} is positive.

%However, in the spirit of \cite{AC}*{Section 6}, we do not expect the infinitesimal deformations coming from torsion sections to be integrable, so one should get local rigidity.

\bibliographystyle{alpha}

\end{document}